\documentclass[12pt]{amsart}

\usepackage{amssymb}

\usepackage[mathscr]{eucal}

\input xypic
\xyoption{all}

\makeindex
\makeglossary

\begin{document}

\baselineskip = 16pt

\newcommand \ZZ {{\mathbb Z}}
\newcommand \NN {{\mathbb N}}
\newcommand \RR {{\mathbb R}}
\newcommand \PR {{\mathbb P}}
\newcommand \AF {{\mathbb A}}
\newcommand \GG {{\mathbb G}}
\newcommand \QQ {{\mathbb Q}}
\newcommand \CC {{\mathbb C}}
\newcommand \bcA {{\mathscr A}}
\newcommand \bcC {{\mathscr C}}
\newcommand \bcD {{\mathscr D}}
\newcommand \bcF {{\mathscr F}}
\newcommand \bcG {{\mathscr G}}
\newcommand \bcH {{\mathscr H}}
\newcommand \bcM {{\mathscr M}}
\newcommand \bcI {{\mathscr I}}
\newcommand \bcJ {{\mathscr J}}
\newcommand \bcK {{\mathscr K}}
\newcommand \bcL {{\mathscr L}}
\newcommand \bcO {{\mathscr O}}
\newcommand \bcP {{\mathscr P}}
\newcommand \bcQ {{\mathscr Q}}
\newcommand \bcR {{\mathscr R}}
\newcommand \bcS {{\mathscr S}}
\newcommand \bcT {{\mathscr T}}
\newcommand \bcV {{\mathscr V}}
\newcommand \bcU {{\mathscr U}}
\newcommand \bcW {{\mathscr W}}
\newcommand \bcX {{\mathscr X}}
\newcommand \bcY {{\mathscr Y}}
\newcommand \bcZ {{\mathscr Z}}
\newcommand \goa {{\mathfrak a}}
\newcommand \gob {{\mathfrak b}}
\newcommand \goc {{\mathfrak c}}
\newcommand \gom {{\mathfrak m}}
\newcommand \gon {{\mathfrak n}}
\newcommand \gop {{\mathfrak p}}
\newcommand \goq {{\mathfrak q}}
\newcommand \goQ {{\mathfrak Q}}
\newcommand \goP {{\mathfrak P}}
\newcommand \goM {{\mathfrak M}}
\newcommand \goN {{\mathfrak N}}
\newcommand \uno {{\mathbbm 1}}
\newcommand \Le {{\mathbbm L}}
\newcommand \Spec {{\rm {Spec}}}
\newcommand \Gr {{\rm {Gr}}}
\newcommand \Pic {{\rm {Pic}}}
\newcommand \Jac {{{J}}}
\newcommand \Alb {{\rm {Alb}}}
\newcommand \alb {{\rm {alb}}}
\newcommand \Corr {{Corr}}
\newcommand \Chow {{\mathscr C}}
\newcommand \Sym {{\rm {Sym}}}
\newcommand \Prym {{\rm {Prym}}}
\newcommand \cha {{\rm {char}}}
\newcommand \eff {{\rm {eff}}}
\newcommand \tr {{\rm {tr}}}
\newcommand \Tr {{\rm {Tr}}}
\newcommand \pr {{\rm {pr}}}
\newcommand \ev {{\it {ev}}}
\newcommand \cl {{\rm {cl}}}
\newcommand \interior {{\rm {Int}}}
\newcommand \sep {{\rm {sep}}}
\newcommand \td {{\rm {tdeg}}}
\newcommand \alg {{\rm {alg}}}
\newcommand \im {{\rm im}}
\newcommand \gr {{\rm {gr}}}
\newcommand \op {{\rm op}}
\newcommand \Hom {{\rm Hom}}
\newcommand \Hilb {{\rm Hilb}}
\newcommand \Sch {{\mathscr S\! }{\it ch}}
\newcommand \cHilb {{\mathscr H\! }{\it ilb}}
\newcommand \cHom {{\mathscr H\! }{\it om}}
\newcommand \colim {{{\rm colim}\, }}
\newcommand \End {{\rm {End}}}
\newcommand \coker {{\rm {coker}}}
\newcommand \id {{\rm {id}}}
\newcommand \van {{\rm {van}}}
\newcommand \spc {{\rm {sp}}}
\newcommand \Ob {{\rm Ob}}
\newcommand \Aut {{\rm Aut}}
\newcommand \cor {{\rm {cor}}}
\newcommand \Cor {{\it {Corr}}}
\newcommand \res {{\rm {res}}}
\newcommand \red {{\rm{red}}}
\newcommand \Gal {{\rm {Gal}}}
\newcommand \PGL {{\rm {PGL}}}
\newcommand \Bl {{\rm {Bl}}}
\newcommand \Sing {{\rm {Sing}}}
\newcommand \spn {{\rm {span}}}
\newcommand \Nm {{\rm {Nm}}}
\newcommand \inv {{\rm {inv}}}
\newcommand \codim {{\rm {codim}}}
\newcommand \Div{{\rm{Div}}}
\newcommand \CH{{\rm{CH}}}
\newcommand \sg {{\Sigma }}
\newcommand \DM {{\sf DM}}
\newcommand \Gm {{{\mathbb G}_{\rm m}}}
\newcommand \tame {\rm {tame }}
\newcommand \znak {{\natural }}
\newcommand \lra {\longrightarrow}
\newcommand \hra {\hookrightarrow}
\newcommand \rra {\rightrightarrows}
\newcommand \ord {{\rm {ord}}}
\newcommand \Rat {{\mathscr Rat}}
\newcommand \rd {{\rm {red}}}
\newcommand \bSpec {{\bf {Spec}}}
\newcommand \Proj {{\rm {Proj}}}
\newcommand \pdiv {{\rm {div}}}
\newcommand \wt {\widetilde }
\newcommand \ac {\acute }
\newcommand \ch {\check }
\newcommand \ol {\overline }
\newcommand \Th {\Theta}
\newcommand \cAb {{\mathscr A\! }{\it b}}

\newenvironment{pf}{\par\noindent{\em Proof}.}{\hfill\framebox(6,6)
\par\medskip}

\newtheorem{theorem}[subsection]{Theorem}
\newtheorem{proposition}[subsection]{Proposition}
\newtheorem{lemma}[subsection]{Lemma}
\newtheorem{corollary}[subsection]{Corollary}

\theoremstyle{definition}
\newtheorem{definition}[subsection]{Definition}

\title{A note on Lefschetz fibrations on algebraic surfaces}

\author{Kalyan Banerjee}


\email{kalyan@math.tifr.res.in, kalyan.b@srmap.edu.in}




\begin{abstract}
Let $S$ be a smooth projective complex algebraic surface and
$f\, :\, S\, \longrightarrow\, {\mathbb C}{\mathbb P}^2$ a finite map. Consider a
pencil of hyperplane sections on ${\mathbb C}{\mathbb P}^2$ and pull it back to $S$. We address the
question whether such a pencil is a Lefschetz pencil on $S$.
\end{abstract}

\maketitle

\section{Introduction}

The theory of algebraic cycles has a fundamental connection with the theory of Lefschetz pencils in the following
sense: Let $S$ be a smooth, projective surface embedded in $\CC\PR^N$. Then consider a Lefschetz pencil on $S$, that
is a fibration of $S$ by hyperplane sections in $\CC\PR^N$, parametrised by a projective line in ${\CC\PR^N}^*$ such
that any singular fiber can have one ordinary double point. Let $0_1,\cdots,0_m$ be the points of
this projective line $\CC\PR^1$ such that the corresponding fibers $S_{0_1},\cdots,S_{0_m}$ are singular.
Take any $t\, \in\, \CC\PR^1$ such that
the corresponding fiber $S_t\, \stackrel{j_t}{\hookrightarrow}\, S$ is smooth. Then by Ehresmann's fibration theorem
$\pi_1(\CC\PR^1\setminus 0_1,\cdots,0_m,t)$ acts on the cohomology $H^1(S_t,\, \QQ)$. This action is given by Picard
Lefschetz formula. Moreover, as a consequence of the Picard Lefschetz formula, the Gysin kernel
$$\ker\{j_{t*}:H^1(S_t,\QQ)\to H^3(S,\QQ)\}$$ is an irreducible submodule of the $\pi_1(\CC\PR^1\setminus
0_1,\cdots,0_m,t)$--module $H^1(S_t,\, \QQ)$.
Now let $A_0(S_t)$ denote the group of
zero cycles on $S_t$, which are algebraically trivial modulo rational equivalence.
The closed embedding $j_t:S_t \,\longrightarrow\, S$ gives rise to a push-forward homomorphism
$j_{t*}$ from $A_0(S_t)$ to $A_0(S)$. A natural question is what is the kernel of this homomorphism $j_{t*}$. The
kernel of $j_{t*}$ is closely related to the monodromy described above. By Mumford-Roitman techniques
\cite{R}, \cite{R1}, \cite{R2}, it can be proved that this kernel is a countable union of translates of an abelian
subvariety of $J(S_t)$, in terms of the identification of $J(S_t)$ with $A_0(S_t)$. Suppose that
$H^3(S,\,\QQ)$ is zero. In this case the Gysin map has $H^1(S_t,\,\QQ)$ as its kernel. The abelian variety mentioned
above gives rise to a Hodge substructure in $H^1(S_t,\,\QQ)$ which is $\pi_1(\CC\PR^1\setminus 0_1,\cdots,0_m,t)$
equivariant. Now by irreducibility of the monodromy action we have this submodule either equal to zero or all of
$H^1(S_t,\,\QQ)$. Consequently the abelian variety is either zero or all of $J(S_t)$. So the kernel of $j_{t*}$ is
either countable or the map is zero. This has been extensively discussed in \cite{BG}.

Our main aim here is to find out when a generic Lefschetz pencil on a smooth projective surface pulls
back to a Lefschetz pencil under a finite map, by using cycle theoretic arguments. More, precisely,
let $f$ be a finite map from a smooth projective surface $S$ to $\CC\PR^2$. Let $\CC\PR^2$ be embedded in
some $\CC\PR^N$. Consider a Lefschetz pencil on $\CC\PR^2$, under this embedding. Pull back the pencil to $S$ by $f$. We
call the pulled back pencil on $S$ to be Lefschetz if the singular fibers have at most ordinary double points,
and ask when the pull-back of a generic Lefschetz pencil a Lefschetz pencil?

We show that for a K3 surface $S$ this is not true. Its proof uses the fact
that a K3 surface has geometric genus greater than zero and hence the albanese map is not an isomorphism. This
leads to following: Let $C_t$ be a hyperplane section of $\CC\PR^2$ and let $C_t'$ be the pull-back of $C_t$ to $S$
by $f$. Consider the closed embedding $j_t:C_t'\to S$. Then it is a consequence of the monodromy of Lefschetz
pencils that kernel of $j_{t*}$ at the level of zero cycles modulo rational equivalence is countable. But this
kernel contains the image of the Jacobian of $C_t$ under $f^*$. This leads to $C_t$ being rational ---
a rather stringent condition. For example we can start with curves in $\CC\PR^N$ of high genus.

Also we discuss the behavior of monodromy under a finite base change for surfaces of general type and threefolds.

The paper is organized as follows. In the next section we recall the basics of monodromy of Lefschetz pencils, and
Picard-Lefschetz formula as present in \cite{Vo}. In the third section we apply these to prove the main
result:

\textit{Let $S$ be a smooth projective surface with geometric genus greater than zero and irregularity equal to
zero. Consider a finite map $f\,:\,S\,\longrightarrow\, \CC\PR^2$. Then a generic pencil on $S$ is not a Lefschetz pencil.}

In  the last two sections we discuss the monodromy of Lefschetz pencils pulled-back under a finite map on threefolds and surfaces of general type with irregularity zero.

{\small \textbf{Acknowledgements:} The author would like to thank the hospitality of Tata Institute Mumbai, for hosting this project. He is indebted to Indranil Biswas for several important discussions about the results present in the manuscript. The author is also grateful to Mahan Mj for helpful discussions relevant to the theme of the paper. Finally the author thanks Vladimir Guletskii for telling the problem to the author and for preliminary discussions which originates the paper. }

\section{The case of a finite map from a smooth projective surface onto $\CC\PR^2$ and Lefschetz fibrations}

Let $S$ be a smooth projective surface over the field of complex numbers. Let $f$ be the finite map from $S$ onto $\CC\PR^2$.
We embed $\CC\PR^2$ into some $\CC\PR^N$ with $N$ big enough by the Veronese embedding. Take a line $\CC\PR^1$ in the dual projective
space ${\CC\PR^N}^*$, and consider the rational map $g\,:\,\CC\PR^2\,\dashrightarrow\, \CC\PR^1$ induced by it.

\begin{definition}
The pencil $g\circ f:S\dashrightarrow \CC\PR^1$ is said to be a {\it Lefschetz pencil} if the only singularities
that occur on the fibers of $g\circ f$ are ordinary double points.
\end{definition}

A natural question is when a generic pencil $g\circ f\,:\, S\,\dashrightarrow\, \CC\PR^1$ a Lefschetz pencil.
For this we need to understand the behavior of monodromy for a pencil
associated to a map $S\,\longrightarrow\,\CC\PR^N$, which is not
an embedding, but finite onto its image with the image being rational.

Let $X$ be a complex manifold of dimension $n$. Let $f\,:\,X\,\longrightarrow\, \Delta$ be a Lefschetz degeneration
over the disk $\Delta$, that is $f$ is proper, and it is smooth
over $\Delta\setminus \{0\}=\Delta^*$ with $X_0\,=\, f^{-1}(0)$ having at most one ordinary double point. Fix $t \, \in\,
\Delta^*$, and let $T\, \in\, \pi_1(\Delta^*, t)\,=\,\ZZ$ be the generator of the fundamental group. Let
$$
\rho\, :\, \pi_1(\Delta^*,\,t)\,\longrightarrow\, \prod_{i=1}^{2n-2} \text{Aut}(H^i(X_t,\,\ZZ))
$$
be the monodromy representation. We recall that
$X_t$ contains a vanishing sphere $S_t^{n-1}$ which is the intersection of a ball centered at the singular point on $X_0$
and $X_t$. The cohomology class of such a sphere is denoted by $\delta$ and it is a generator of
$$\ker(H^{n-1}(X_t,\ZZ)\,\longrightarrow\, H^{n+1}(X,\ZZ))\;.$$
Also note that $X_t$ is real oriented $2n-2$ manifold with an intersection form $\langle-,-\rangle$ on $H^{n-1}(X_t,\,\ZZ)$.

The following theorem is proved in \cite{Vo}:

\begin{theorem}[{\cite[Ch.~3, Theorem 3.16]{Vo}}]\label{theorem1}
For every $\alpha$ in $H^{n-1}(X_t,\,\ZZ)$ the action of $T$ on $\alpha$ is given by
$$T(\alpha)\,=\,\alpha\pm\langle \alpha, \delta\rangle \delta\;.$$
\end{theorem}

As before, take
$$
f\, :\, S \,\longrightarrow\,  \CC\PR^2 \, \hookrightarrow\, \CC\PR^N\, ,
$$
where the map $f$ is finite, and the embedding of $\CC\PR^2$ is given by a Veronise map. Let
$\bcD_S$ denote the discriminant variety of $S$ in ${\CC\PR^N }^{*}$ consisting of hyperplanes whose pullback to
$S$ is singular. Let
\begin{equation}\label{deu}
U\,:=\,{\CC\PR^N}^{*}\setminus \bcD_S\, ,
\end{equation}
so $U$ parametrizes the hyperplanes whose pull-back
under $f$ is smooth. So we have the family:
$$\bcS_U\,:=\,\{(x,\,H)\,\in\, S\times U\, \mid\, f(x)\,\in \, \CC\PR^2\cap H\}\, ;$$
the natural projection $\pr_2\, :\, \bcS_U\,\longrightarrow\, U$ is a submersion with fibers equal to
$S\cap f^{-1}(\CC\PR^2\cap H)$.

The discriminant variety $\bcD_S$ is a Zariski closed subset of ${\CC\PR^N}^*$. Now observe that $\bcD_S$ is contained in
 the discriminant variety of $\CC\PR^2$ in $\CC\PR^N$, denoted by $\bcD$.

\begin{lemma}\label{theorem3}
Let ${\mathcal P}\, \subset\, {\CC\PR^N}^*$ be a projective line passes through a fixed point $0$ of $U$ in \eqref{deu}
and intersects $\bcD_S$ transversally in its smooth locus. Then the natural homomorphism
$$\pi_1({\mathcal P}\setminus ({\mathcal P}\cap \bcD_S),\,0)\,\longrightarrow\, \pi_1(U,\,0)$$
is surjective.
\end{lemma}

\begin{proof}
First identify $\CC\PR^N$ with its dual. Let $W$ be the set of lines in ${\CC\PR^N}$ which passes through $0$ and meets
$\bcD_S$ transversally in its smooth
locus. For
$$P\, :=\, \{(x,\,t)\,\in\, \CC\PR^N\times W\, \mid\, x\,\in\, {\mathcal P}_t\}\, ,$$
the natural projection $\pr_2\,:\,P\,\longrightarrow\, W$ is a $\CC\PR^1$ bundle. Consider the first projection
$\pr_1\,:\,P\,\longrightarrow\, \CC\PR^N$, and set $U'\,=\,\pr_1^{-1}(U)\,=\, P\setminus \pr_1^{-1}(\bcD_S)\;. $
By the definition of $W$, the restriction of $\pr_2$ to $U'$ is a fibration. Also since all lines in $W$ pass
through $0$, the restriction $\pi\,=\,{\pr_2}|_{U'}$ admits a section given by $t\,\longmapsto\, (0,\,t)$. The image
of this section will be denoted by $W'$. Let $0'\, =\, (0,\, o)$ be a point over $0$. Note that the fibers of $\pi$ are identified with $\mathcal P\setminus \mathcal P\cap \bcD_S$, we have a homotopy exact sequence:
$$\pi_1(\mathcal P\setminus \mathcal P\cap\bcD_S,0)\, \longrightarrow\, \pi_1(U',0')\, \longrightarrow\, \pi_1(W,o)
\, \longrightarrow\, 1\;.$$
We can apply $pr_{1*}$ to this exact sequence. $W'\subset U'$ gives a section of $\pi$, so we have
$$\pi_1(W',0')\subset \pi_1(U',0')\;.$$
Now $\pr_{1}$ contracts $W$ to $0$ and we have
$$\pr_{1*}(\pi_1(\mathcal P\setminus \mathcal P\cap \bcD_S,0))$$
surjects onto $\pr_{1*}(\pi_1(U',0'))$. So it remains to show that $\pr_{1*}$ from $\pi_1(U',0')$ surjects onto $\pi_1(U,0)$. Observe that $\pr_1$ is birational. Consider the rational map from $U$ to $U'$ constructed in the following way. Let $x$ be a point in $U$ different from $0$, then there exists a unique line joining $x$ and $0$;
call this line $\mathcal P_t$. Let $t$ be the point in $W$ corresponding to $\mathcal P_t$. Send $x$ to $(x,t)$. This map is well-defined on the complement of Zariski closed subset of $U$, that is those points such that the line joining $x,0$ does not meet $\bcD_S$ transversally.  Let $U''$ be the complement of this set in $U'$. Then the real codimension of $U'\setminus U''$ is at least $2$. Let $x$ be a point in $U$ other than $0$ and $x'$ its pre-image in $U'$. By the above fact, it follows that $\pi_1(U'',x')$ surjects onto $\pi_1(U',x')$. Let $V''$ be the Zariski open set of $U$ isomorphic to $U''$. Since the real codimension of $U\setminus V''$ is at least $2$, we have $\pi_1(V'',x)$ surjects onto $\pi_1(U,x)$  Therefore we have a commutative diagram as follows:

$$
  \diagram
 \pi_1(U'',x')\ar[dd]_-{} \ar[rr]^-{} & & \pi_1(U',x') \ar[dd]^-{} \\ \\
  \pi_1(V'',x) \ar[rr]^-{} & & \pi_1(U,x)
  \enddiagram
  $$
Here the left vertical map is an isomorphic and the bottom horizontal map is surjective. Therefore it follows that
$\pi_1(U',x')$ surjects onto $\pi_1(U,x)$. Since $U$ is path connected it follows that $\pi_1(U',0')$ surjects
onto $\pi_1(U,0)$.
\end{proof}

 Let $\bcD_S^0$ be the Zariski open subset of $\bcD_S$ parametrising the hyperplanes in $\CC\PR^N$ whose pull-back
to $S$ has one ordinary double point. Consider the fibration $\pr_2\,:\,\bcS_U\, \longrightarrow\, U$; for any
$z\, \in\, U$, the fiber ${\rm pr}^{-1}_2(z)$ will be denoted by $S_z$.
Fix $0\,\in\, U$; we have the monodromy representation
$$\rho\,:\,\pi_1(U,\,0)\, \longrightarrow\, Aut(H^{1}(S_0,\,\ZZ))\, .$$
For any $y$ in $\bcD_S^0$, let $y'$ in $U$ be near $y$, in the sense that it is contained in a disc $D_y$, which meets $\bcD_S^0$ transversally at $y$ and $D_y\setminus \{y\}$ is in $U$. This can be done by choosing a projective line through $y$ and $y'\in U$, such that the line intersect $\bcD_S^0$ transversally. Then we have a vanishing cycle
$\delta_y$ well-defined up to a sign as a generator of the kernel of the homomorphism
$$H_{n-1}(S_{y'},\,\ZZ)\, \longrightarrow\, H_{n-1}(S_{D_{y}},\,\ZZ)\, ,$$
where $S_{D_y}\, :=\, \pr_2^{-1}(D_y)$, \cite[Lemma 2.26]{Vo}. Let $\gamma$ be a path from $y'$ to $0$ in $U$. Then
trivializing the fibration over the image of $\gamma$, we can construct a diffeomorphism, \cite[subsection 3.1.2]{Vo},
$$\phi\,:\,S_{y'}\,\stackrel{\sim}{\longrightarrow}\, S_0$$
which is unique up to homotopy. Consider $\phi_*(\delta_y)$ in $H_{1}(S_0,\,\ZZ)\,=\,H^{1}(S_0,\,\ZZ)$. This process could be done for any arbitrary point of $\bcD_{S}^0$. Therefore given any $y$ in $\bcD_{S}^0$, we have a vanishing cycle $\phi_*(\delta_y)$ in $H^1(S_0,\ZZ)$, where $\phi$ is the diffeomorphism between $S_{y'}$ and $S_0$ as described above. Notice that this vanishing cycle is dependent on the choice of $y'$ and the path between $y'$ and $0$. But the homology class $\phi_*(\delta_y)$ defined by the choice of an orientation of $S_0$ is well defined upto sign \cite{Vo}[chapter 2, section 2.2.1].

\begin{theorem}
\label{theorem4}
All the vanishing cycles constructed above are conjugate under the monodromy action.
\end{theorem}

\begin{proof}
By definition of the monodromy action if we change the path $\gamma$ by a loop $\gamma'$ at $0$, then the morphism
$\phi_*$ is nothing but $\rho(\gamma')\circ\phi_*$. Therefore
$$\delta_{\gamma'.\gamma}=\phi_*(\rho(\gamma')\delta_y)=\rho(\gamma')(\delta_{\gamma})\;.$$
So it is sufficient to check  that $\delta_{y}$ is conjugate to $\delta_{z}$ when we change the point $y$ to $z$. Note that $\bcD_S^0$ (the Zariski open subset of $\bcD_S$ parametrising the pull-back of hyperplane sections with one ordinary double point singularity), is a subset of $\bcD_0$, parametrising the hyperplane sections of $\CC\PR^2$, with one ordinary double points. $\bcD_0$  is smooth and hence path-connected, \cite{Vo}[ chapter 2, lemma 2.7]. Let $y,z$ be two points in $\bcD_S^0$. Choose a path $l$ from $y$ to $z$ in $\bcD_0$. Now by path-lifting we can lift $l$ to some path $l'$ from $y'$ to $z'$ in the complement of $\bcD_0$, which is a subset of $U$. Consider a trivialisation of $\pr_2$ over $l'$. That gives us a diffeomorphism of $S_{y}$ and $S_{z}$. Hence $\delta_{y}$ is transported to $\delta_{z}$. Let $\gamma$ be a path from $0$ to $y'$ and $\gamma'$ is path from $0$ to $z'$. Consider the loop
$$\gamma''=\gamma'^{-1}.l'.\gamma$$
based at $0$ and it satisfies
$$\rho(\gamma'')(\delta_{\gamma})=\delta_{\gamma'}\;.$$
\end{proof}
Let $j$ denote the embedding of $S_0$ into $S$, for a fixed point $0$ in $U$.
Now consider the embedding of $\bcS_U$ into $S\times U$, call it $J$. Then this $J$ induces a push-forward homomorphism of local systems
$$J_*:R^{1}\pr_{1*}\QQ\to R^{3}\pr_{1*}\QQ$$
whose value at the stalk over $0$, is nothing but the kernel
$$\ker(j_*:H^1(S_0,\QQ)\to H^3(S,\QQ))\;.$$
Kernel of $J_*$ defines a local system, hence $\pi_1(U,0)$ acts on $\ker(j_*)$.

\begin{theorem}
\label{theorem8}
The action of $\pi_1(U,0)$ on the vanishing cohomology
$$\ker(j_*:H^1(S_0,\QQ)\to H^3(S,\QQ))\;.$$
is irreducible, if there exists a Lefschetz pencil $\mathcal P$ in ${\CC\PR^N}^*$ intersecting $\bcD_S$ transversely.
\end{theorem}
\begin{proof}
By Lemma \ref{theorem3}, we have that the action of $\pi_1(U,0)$ can be computed by considering a Lefschetz pencil $\mathcal P$ in ${\CC\PR^N}^*$, which intersect $\bcD_S$ transversely. By lemma 2.26 in \cite{Vo} the vanishing cohomology is generated by vanishing cycles and by proposition 2.27 \cite{Vo}, the intersection form restricted to $\ker(j_*)$ is non-degenerate. Now $\pi_1(\mathcal P\setminus\{0_1,\cdots,0_m\})$ is generated by $\gamma_1,\cdots,\gamma_m$, such that the action of $\gamma_i$ is given by
$$\rho(\gamma_i)(\alpha)=\alpha\pm \langle \alpha,\delta_{\gamma_i}\rangle \delta_{\gamma_i}\;,$$
here $\delta_{\gamma}$ is the vanishing cycle corresponding to the loop $\gamma_i$.
Here $\gamma_i$ arises from a path from $0$ to $0_i$, $0_i$'s correspond to the singular fibers.
Suppose that $V$ is a $\pi_1(\mathcal P\setminus\{0_1,\cdots,0_m\})$ stable subspace in $\ker(j_*)$. Also suppose that $V$ is non-trivial. Let $\alpha$ be a non-zero element in $V$. Then
$$\rho(\gamma_i)(\alpha)=\alpha\pm \langle \alpha, \delta_{\gamma_i}\rangle\alpha$$
Since $\alpha$ is non-zero and the intersection form restricted to $\ker(j_*)$ is non-degenerate, there exists $\gamma_i$, such that
$\langle \alpha, \delta_{\gamma_i}\rangle$ is nonzero. By the above formula it follows then that $\delta_{\gamma_i}$ is in $V$, for some $i$. Since $V$ is stable under the action of $\pi_1(\mathcal P\setminus\{0_1,\cdots,0_m\})$ and all the vanishing cycles are conjugate, all the vanishing cycles belong to $V$. Hence $V$ is $\ker(j_*)$.
\end{proof}

\section{Application of the above result}
Let $S$ be a surface with geometric genus greater than $0$ and irregularity $0$. Consider as before a finite map $f:S\to \CC\PR^2$. Can we say that a generic pencil in ${\CC\PR^N}^{*}$ pulled back under $f$ is a Lefschetz pencil. For that we consider the group of algebraically  trivial zero cycles on $S$, denote it by $A_0$. Suppose that we have a hyperplane section $C_t$ in $\CC\PR^2$, whose inverse image is smooth, we denote it by $\wt{C_t}$. Then we have the Cartesian square:

$$
  \diagram
 \wt{C_t}\ar[dd]_-{} \ar[rr]^-{} & & C_t \ar[dd]^-{} \\ \\
  S \ar[rr]^-{} & & \CC\PR^2
  \enddiagram
  $$
This gives rise to the following square at the level of zero cycles
$$
  \diagram
 A_0(C_t)\ar[dd]_-{} \ar[rr]^-{} & & A_0(\wt{C_t}) \ar[dd]^-{j_{t*}} \\ \\
  A_0(\CC\PR^2) \ar[rr]^-{} & & A_0(S)
  \enddiagram
  $$
Since $A_0(\CC\PR^2)$ is zero, we have that $A_0(C_t)\cong J(C_t)$ maps into the kernel of $j_{t*}$. Suppose that there exists a Lefschetz pencil $\mathcal P$ through $t$ intersecting $\bcD_S$ transversely. Then we have that $\ker(j_{t*})$ at the level of cohomology is irreducible under the action of $\pi_1(\mathcal P\setminus\{0_1,\cdots,0_m\}$. From this we deduce the following:

\begin{theorem}
\label{theorem5}
The abelian variety $J(\wt{C_t})$ is simple.
\end{theorem}

\begin{proof}
Suppose that there exists an abelian subvariety $A_t$ of $J(\wt{C_t})$, which is proper and non-trivial. Extending the base field to $\CC(t)$, we have $A_t,J(\wt{C_t})$ can be viewed as abelian varieties over $\Spec(\CC(t))$. They have a minimal field of definition $L$, which is a finite extension of $\CC(t)$ in $\overline{\CC(t)}$. Choose a smooth projective curve $\mathcal P'$, such that $\CC(\mathcal P')=L$. Then we spread out $A_t,J(\wt{C_t})$ over some Zariski open set $U'$ of $\mathcal P'$. Let us denote the spreads as $\bcA,\bcJ$. Throwing out some more points from $U'$ we can assume that $\bcA\to U',\bcJ\to U'$ are smooth in the complex analytic sense. Also they are proper as the morphisms are locally projective. Therefore by Ehresmann's theorem the $f:\bcA\to U',g:\bcJ\to U'$ are fibrations in the complex analytic sense. Therefore the higher direct images $R^if_*\QQ,R^ig_*\QQ$ are locally constant sheaves and they give rise to local systems. Hence we have $\pi_1(U',0')$ action on $H^{2d-1}(A_t,\QQ)$ and on $H^{2g-1}(J(\wt{C_t}),\QQ)$, where $d=\dim(A_t)$ and $g$ is the genus of $\wt{C_t}$. The later cohomology group is isomorphic to $H^1(\wt{C_t},\QQ)$ in a natural way.  So we have an inclusion of $H^{2d-1}(A_t,\QQ)$ into $H^1(\wt{C_t},\QQ)$, and it is induced by a morphism of local systems whence a map of $\pi_1(U',0')$ modules.  Now let $U'$ be the inverse image of an open set $U$ of $\mathcal P$. Then  the map from $U'$ to $U$ is finite, hence it gives rise to a finite index subgroup $\pi_1(U',0')$ in $\pi_1(U,0)$. Let us show that $H_t=H^{2d-1}(A_t,\QQ)$ is $\pi_1(U,0)$- equivariant. Let $\gamma$ be in $\pi_1(U,0)$, so that $\gamma^m$ is in $\pi_1(U',0')$. Then we have (by the Picard Lefschetz formula),
$$\rho(\gamma^m)(\alpha)=\alpha\pm m\langle \alpha,\delta_{\gamma}\rangle \delta_{\gamma}$$
since $\rho(\gamma^m)(\alpha)-\alpha$ belong to $H_t$, we have that $\langle \alpha,\delta_{\gamma}\rangle \delta_{\gamma}$ belongs to $H_t$. Hence by Picard-Lefschetz formula $H_t$ is $\pi_1(U,0)$ equivariant. Since $H^1(\wt{C_t},\QQ)$ is $\pi_1(U,0)$ irreducible module, we have $H_t=0$ or all of $H^1(\wt{C_t},\QQ)$. Hence $A_t=\{0\}$ or $J(\wt{C_t})$.
\end{proof}

Now the kernel of $j_{t*}$ can be written as follows:

\begin{theorem}
\label{theorem6}

The kernel of $j_{t*}$ is a countable union of translates of an abelian subvariety of $J(\wt{C_t})$.
\end{theorem}

\begin{proof}

 we have the following commutative diagram:

$$
  \diagram
 \Sym^g \wt{C_t}\ar[dd]_-{\theta_t} \ar[rr]^-{j_t} & & \Sym^g S \ar[dd]^-{\theta_S} \\ \\
  J(\wt{C_t}) \ar[rr]^-{j_{t*}} & & A_0(S)
  \enddiagram
  $$
Hence we can write $\ker(j_{t*})=\theta_t(j_t^{-1}\theta_S^{-1}(0))$. The fiber of $\theta_S$ is a countable union
of Zariski closed subsets in $\Sym^g S$ \cite{Vo}[Volume II, 10.7]. Therefore the kernel of $j_{t*}$ is a
countable union of Zariski closed subsets of $J(\wt{C_t})$. Since we work over $\CC$ which is uncountable, it is
possible to prove that there exists a unique Zariski closed subset of maximal dimension which passes through the
identity element and sits inside $\ker(j_{t*})$. Hence it will be an abelian variety. Therefore the kernel of
$j_{t*}$ is a countable union of shifts of an abelian variety $A_t$ in $J(\wt{C_t})$.
\end{proof}

\begin{theorem}
Let $S$ be a smooth projective surface with geometric genus greater than zero and irregularity equal to zero. Consider a finite map $f:S\to \CC\PR^2$. Then a generic pencil on $S$ is not a Lefschetz pencil.
\end{theorem}
\begin{proof}
By the above, the abelian variety $A_t$ is either trivial or all of $J(\wt{C_t})$. Now suppose that a generic pencil $\mathcal P$ is a Lefschetz pencil on $S$. Then through a general point $t$ in ${\CC\PR^N}^*$ we have a Lefschetz pencil. Therefore $A_t$ is either trivial or $J(\wt{C_t})$ for all general $t$. Note that $A_t$ gives rise to a sub-variation of Hodge structures of the variation of Hodge structures coming from the Jacobian fibration of the family of curves $\wt{C_t}$ on $S$. Therefore either for all $t$ general $A_t=0$ or it is all of $J(\wt{C_t})$. In the second case we have that $A_0(S)$ is $0$. Indeed given any two points we can find a $\wt{C_t}$, which passes through the two points. Then $j_{t*}(P-Q)$ is $0$. Since the group $A_0(S)$ is generated by cycles of the form $j_{t*}(P-Q)$, the group itself is zero. Therefore for a general $t$, the kernel of $j_{t*}$ is countable. But according to the above it contains the image of  $J(C_t)$. Hence it follows that $C_t$ is rational. This is absurd because  we could start we a very ample line bundle on $\CC\PR^2$ such that the curves in the linear system are non-rational. Hence the kernel cannot be countable. So it all means that a generic pencil on $S$ is not a Lefschetz pencil.
\end{proof}

\section{Lefschetz pencils on Enriques Surfaces}

In this section we want to study the case of an Enriques surface. So let $S$ be a K3 surface with an anti-symplectic, fixed point free involution $i$ on it. Then $S/i=S'$ is an Enriques surface. Consider a finite $f$ map from $S'$ to $\CC\PR^2$, where $\CC\PR^2$ is embedded in some $\CC\PR^N$. Then consider a pencil of hyperplanes on $\CC\PR^N$, and pull it back to $S'$. The question is for a generic choice of the pencil, can it be a Lefschetz pencil?

For that we consider the composition $g:S\to S'\to \CC\PR^2$. Let $\mathcal P$ be a pencil of hyperplanes in $\CC\PR^N$. Let for $t$ in $\mathcal P$, $C_t$ is the pull-back of a hyperplane under $f$ and $\wt{C_t}$ is its double cover on $S$. Then we have the following Cartesian square.
$$
  \diagram
 \wt{C_t}\ar[dd]_-{} \ar[rr]^-{} & & C_t \ar[dd]^-{} \\ \\
  S \ar[rr]^-{} & & S'
  \enddiagram
  $$
This gives the following commutative square at the level of algebraically trivial zero cycles.
$$
  \diagram
 A_0(C_t)\ar[dd]_-{} \ar[rr]^-{} & & A_0(\wt{C_t}) \ar[dd]^-{j_{t*}} \\ \\
  A_0(S') \ar[rr]^-{} & & A_0(S)
  \enddiagram
  $$
By \cite{BKL}, the group $A_0(S')$ is trivial, hence the image of $J(C_t)$ is contained in the kernel of the push-forward $j_{t*}$ from $J(\wt{C_t})$ to $A_0(S)$. Suppose that a generic pencil pulled back to $S'$ is Lefschetz. Then by Theorem \ref{theorem5}, $J(C_t)$ is a simple abelian variety and it is contained in a countable union of translates of an abelian variety $A_t$ in $\ker(j_{t*})$, by Theorem \ref{theorem6}. Now consider the map at the level of cohomology from $H^1(C_t,\QQ)$ to $H^1(\wt{C_t},\QQ)$. Since $S\dashrightarrow \mathcal P$ gives rise to a fibration, by blowing up the base locus of this rational map, we have $\pi_1(\mathcal P\setminus 0_1,\cdots,0_m)$ acts on $H^1(\wt{C_t},\QQ)$. On  the other hand the action of $\pi_1(\mathcal P\setminus 0_1,\cdots,0_m)$ on $H^1(C_t,\QQ)$ is irreducible by Theorem \ref{theorem8}. Hence the image of $H^1(C_t,\QQ)$ in $H^1(\wt{C_t},\QQ)$ has an action of $\pi_1(\mathcal P\setminus 0_1,\cdots,0_m)$, which is irreducible in the sense that it has no non-trivial and proper $\pi_1(\mathcal P\setminus 0_1,\cdots,0_m)$ equivariant subspace. Denote the image of $H^1(C_t,\QQ)$ in $H^1(\wt{C_t},\QQ)$ by $W$ and its orthogonal complement under the bilinear form on $H^1(\wt{C_t},\QQ)$ by $W^{\perp}$. We prove the following:

\begin{proposition}
$W^{\perp}$ is equivariant under the action of the fundamental group $\pi_1(\mathcal P\setminus 0_1,\cdots,0_m)$ and is irreducible.
\end{proposition}

\begin{proof}
Let $w$ belongs to $W^{\perp}$, then for all $v$ in $W$, we have $\langle v,w \rangle=0$. We have to prove that for all $\gamma$ in $\pi_1(\mathcal P\setminus 0_1,\cdots,0_m)$, $\gamma.w$ is again in $W^{\perp}$. Observe that we have
$$\langle v,\gamma.w\rangle=\langle \gamma.v', \gamma.w \rangle=\langle v',w\rangle=0$$
this happens for all $v'$. Hence $\gamma.w$ is in $W^{\perp}$ for all $\gamma$.

Now suppose that $W'$ is a $\pi_1(\mathcal P\setminus 0_1,\cdots,0_m)$ equivariant subspace of $W^{\perp}$. Then consider the subspace of $W$, given by
$$\{v\in W|\langle v,w\rangle=0, \forall w\in  W'\}\;.$$
We prove that $V$ is $\pi_1(\mathcal P\setminus 0_1,\cdots,0_m)$ equivariant. Let $\gamma$ belongs to $\pi_1(\mathcal P\setminus 0_1,\cdots,0_m)$. Then we have
$$\langle \gamma.v,w\rangle=\langle \gamma.v, \gamma.w'\rangle=\langle v,w'\rangle=0\;.$$
This happens for all $w'$ in $W'$. Hence $\gamma.v$ is in $V$, for all $v$ in $V$ and for all $\gamma$ in $\pi_1(\mathcal P\setminus 0_1,\cdots,0_m)$. Therefore it follows by Picrad-Lefschetz formula that $V=0$ or $W$. Consequently $W'=W$ or $0$. Note that here the fact that $W$ is finite dimensional is necessary.
\end{proof}

\begin{theorem}
The pull-back of a generic pencil on $\CC\PR^2$ to $S'$ is not of Lefschetz type.
\end{theorem}

\begin{proof}
Now it follows from the above proposition that the abelian variety $J(C_t)^{\perp}$ in $J(\wt{C_t})$ is also simple. So consider the kernel of $j_{t*}$, which is a countable union of translates of an abelian subvariety $A_t$ of $J(\wt{C_t})$. We have $A_{t}^{\perp}$ is inside $J(C_t)^{\perp}$. Since $J(C_t)^{\perp}$ is simple we have $A_t^{\perp}$ is trivial or all of $J(C_t)^{\perp}$. Consequently we have $A_t$ is isogenous to $J(\wt{C_t})$ or it is isogenous to $J(C_t)$. In the first case all of $J(\wt{C_t})$ goes to zero. Given any two points on $S$, there exists a $\wt{C_t}$, which passes through this two points. Since $J(\wt{C_t})$ goes to zero under $j_{t*}$, for any $t$ such that $\wt{C_t}$ is smooth, we have $A_0(S)$ is zero. This is a contradiction to the fact that $A_0(S)$ is not isomorphism to the albanese variety of $S$. Therefore we have that $A_t$ is isogneous to $J(C_t)$. From this it follows that all elements in the kernel are invariant under the action of the involution $i$. On  the other and the group $A_0(S)$ has the involution acting as $-1$ on the Chow group. So all elements in the kernel of $j_{t*}$ is 2-torsion, so by Roitman's theorem the kernel is trivial. Hence $C_t$ is rational. But this may not be true because we can take the pull-back of a very ample divisor on $\CC\PR^2$ to $S'$. By adjunction formula and ampleness criterion it will follow that a curve in the corresponding linear system is non-rational.
\end{proof}

On the other hand we have $A_0(S')=0$, implying that the $i$-invariant part of $A_0(S)$ is trivial. So it follows
that all elements of $A_0(S)$ are $i$-antiinvariant. That is for $z$ in $A_0(S)$ we have $i_*(z)=-z$. Since the
geometric genus of the surface $S$ is greater than zero, the group $A_0(S)$ is not isomorphic to the albanese
variety of $S$. Or in other words the natural map from the symmetric powers of $S$ to $A_0(S)$ are never
surjective. We also have the following notion of finite dimensionality of a subgroup of $A_0(S)$, due to Roitman
\cite{R1}.

\begin{definition}
A subgroup $P$ of $A_0(S)$ is finite dimensional, if there exists a smooth projective variety $W$ and a correspondence $\Gamma$ on $W\times S$, such that $P$ is contained in the set $\{\Gamma_*(w):w\in W\}$.
\end{definition}

Let $S$ be a smooth projective surface. Suppose that $Z$ is a correspondence on $S\times S$ such that image of
$Z_*$ from $A_0(S)$ to $A_0(S)$ is finite dimensional. Then we have the following result due to Voisin \cite{Voi}.

\begin{theorem}[{\cite{Voi}[Theorem 2.3]{Voi}}]\label{theorem7}
Let $Z$ be a correspondence on $S\times S$ such that image of $Z_*$ from $A_0(S)$ to $A_0(S)$ is finite dimensional in the above sense. Then $Z_*$ factors through the albanese variety of $S$.
\end{theorem}

With this information at hand, we prove the following theorem (the proof of this theorem goes along the line of
proof of Proposition 2.6 in \cite{Voi}):

\begin{theorem}
Let $S$ be a K3 surface with an anti-symplectic involution $i$, so that $S/i$ is an Enriques surface $S'$. Then there does not exist a rational curve on $S$ which is ample and invariant under the involution, or passes through one of the fixed points of the involution.
\end{theorem}

\begin{proof}
Fix some very ample line bundle $L$ on $S'$. Then by adjunction formula and the fact that $2K_{S'}$ is trivial, we have
$$L^2\,=\,2g-2$$
where $g$ is the genus of a curve in the linear system of $L$. By Riemann-Roch we have that $H^0(C,L|_{C})$ is $g$ dimensional. Consider the exact sequence of line bundles:

$$0\to\bcO(C)\to \bcO(S')\to \bcO(S')/\bcO(C)\to 0$$

tensoring with $\bcO(-C)$ we get

$$0\to \bcO(S')\to \bcO(-C)\to \bcO(-C)|_{C}$$

Therefore we have the long exact sequence at the level of cohomology of the above sheaves. This gives a complex
$$0\, \longrightarrow\, \CC\, \longrightarrow\, H^0(S',\, L)\, \longrightarrow\, H^0(C,\,L|_C)\, \longrightarrow\, 0\, .$$
The sequence is exact on the right because $H^1(S',\,\bcO(S'))$ is zero as irregularity of $S'$ is zero.
Hence the dimension of $H^0(S',L)$ is $g+1$. Now consider a general curve $C$ in the linear system $|L|$. The
inverse image of $C$, say $\wt{C}$ is connected. Indeed, suppose that $\wt{C}$ has two components $C_1,C_2$. Then
$C_1^2>0,C_2^2>0$ and $C_1.C_2=0$. This means by the Hodge index theorem that $C_1=rC_2$, for some integer $r$.
Hence $C_1^2=0=C_2^2$. Therefore $(C_1+C_2)^2=0$, contradicting that $C_1+C_2$ is ample.

Now consider a general element of $S^g$, say $(s_1,\cdots,s_g)$ such that all $s_i$'s are distinct. The images of $s_i$'s under the quotient map from $S$ to $S'$ are contained in a unique curve $C$, if all the images of $s_i$'s are distinct. Let $\wt{C}$ be the curve containing $s_i$, for all $i$. Consider the map
$$S^g\, \longrightarrow\, A_0(S)$$
given by
$$(s_1,\,\cdots,\,s_g)\,\longmapsto \,\sum_j (s_j-i(s_j))\;.$$
The above tell us that the map is factoring through A natural map from $S^g$ to $\bcP(\wt{\bcC}/\bcC)$. Here $\bcC,\wt{\bcC}$, are the universal family of curves in the linear system $|L|$, the universal family of the double covers of the curves in the linear system respectively. $\bcP$ is the Prym fibration.
So given a tuple $(s_1,\,\cdots,\,s_g)$ consider the element $\sum_j\alb_{\wt{C}}(s_j-i(s_j))$  in the prym variety $P(\wt{C}/C)$. Now the dimension of $\bcP(\wt{\bcC}/\bcC)$ is $2g-1$, since dimension of $|L|$ is $g$ and dimension of the prym variety $P(\wt{C}/C)$ is $g-1$ (by the Riemann-Hurwitz formula). So  the fibers of the map $S^g\to \bcP(\wt{\bcC}/\bcC)$ are positive dimensional. Hence the fiber of the map  from $S^g$ to $A_0(S)$ are positive dimensional. Denote the fiber by $F_s$.

Suppose that there exists a rational curve $R$ on $S$, which is ample and invariant under $i$ or it passes through one of the fixed points of the involutions. Consider the divisor $\sum_i \pr_i^{-1}R$ on $S^g$. It is ample. So it intersects the curve $F_s$. So there exists a point $p$ on $R$ such that $i(p)$ is rationally equivalent to $p$  and $(p,s_1,\cdots,s_{g-1})$ is in $F_s$. So it all means that $\Gamma_*(S^g)=\Gamma_*(S^{g-1})$, where $\Gamma=\Delta_S-\Gr(i)$.

Now we prove by induction that $\Gamma_*(S^{g-1})=\Gamma_*(S^m)$ for all $m\geq g$.
So suppose that $\Gamma_*(S^k)=\Gamma^*(S^{g-1})$ for $k\geq g$, then we have to prove that $\Gamma_*(S^{k+1})=\Gamma_*(S^{g-1})$. So any element in $\Gamma_*(S^{k+1})$ can be written as  $\Gamma_*(s_1+\cdots+s_{g-1})+\Gamma_*(s)$. Now let $k-g+1=l$, then $g=k-l+1$. Since $k-l<k$, we have $k-l+1\leq k$, so $g\leq k$, so we have the cycle
$$\Gamma_*(s_1+\cdots+s_{g-1})+\Gamma_*(s)$$
supported on $S^k$, hence on $S^{g-1}$. So we have that $\Gamma_*(S^{g-1})=\Gamma_*(S^k)$ for all $k$ greater or equal than $g$. Now any element $z$ in $A_0(S)$, can be written as a difference of two effective cycle $z^+,z^-$ of the same degree. Then we have
$$\Gamma_*(z)=\Gamma_*(z^+)-\Gamma_*(z_-)$$
and $\Gamma_(z_{\pm})$ belong to $\Gamma_*(S^{g-1})$. So let $\Gamma'$ be the correspondence on $S^{2g-2}\times S$ defined as
$$\sum_{l\leq g-1}(pr_i,pr_{S})^*\Gamma-\sum_{g-1\leq l\leq 2g-2}(pr_i,pr_{S})^* \Gamma$$
where $\pr_i$ is the $i$-th projection from $S^{g-1}$ to $S$, and $\pr_{S}$ is from $S^i\times S$ to the last copy of $S$. Then we have
$$\im(\Gamma_*)=\Gamma'_*(S^{2g-2})\;.$$
This prove that image of $\Gamma_*$ is finite dimensional. But then by the previous Theorem \ref{theorem7}, we have that $i_*=\id$ on $A_0(S)$. But we know that $i_*=-\id$ on $A_0(S)$. Therefore $2z=0$, for all $z$ in $A_0(S)$. By the Roitman's torsion theorem we have that $z=0$, \cite{R2}. Hence $A_0(S)=0$, contradicting the fact that $A_0(S)$ is not isomorphism to the albanese variety of $S$.
\end{proof}

\section{The case of surfaces of general type with irregularity zero appearing as branched double covers of rational surfaces}

Let $S$ be a surface of general type with irregularity zero. Suppose that there exists an involution $i$ on $S$ such that $S/i$ is birational to the projective plane or to an Enriques surface. Examples of such surfaces are of the type numerical Godeaux or numerical Campedelli, that is when the geometric genus of $S$ is zero and self intersection of the canonical bundle of $S$ is either $1$ or $2$. Consider the isolated fixed points of the involution. Blow up $S$ along these points. Call the blow-up as $\wt{S}$. There is an involution on $\wt{S}$. The quotient of $\wt{S}$ by this involution is either a rational surface or an Enriques surface. Call it $\wt{S}/i$. Then consider a very ample line bundle on $\wt{S}/i$, that is we embed $\wt{S}/i$ in some projective space $\CC\PR^N$. Consider a generic pencil on $\wt{S}/i$ and pull it back to $\wt{S}$. We can ask whether the pull-back of a generic pencil to $\wt{S}$ is Lefschetz or not. So we prove the following theorem:

\begin{theorem}
Let $\wt{S}, \wt{S}/i$ be as above. Suppose that the pull-back of a generic pencil on $\wt{S}/i$ to $\wt{S}$ is a Lefschetz pencil. Then $\wt{S}$ admits of a hyper-elliptic fibration.
\end{theorem}

\begin{proof}
So let $\wt{S}/i\to \mathcal P$ be a Lefschetz pencil, that is the singular fibers of this fibration can have at most one ordinary double point. Consider the composition $\wt{S}\to \wt{S}/i\to \mathcal P$. Let $t$ be a closed point on $\mathcal P$. Consider the fiber-product diagram,
$$
  \diagram
 \wt{C_t}\ar[dd]_-{} \ar[rr]^-{} & & C_t \ar[dd]^-{} \\ \\
  \wt{S} \ar[rr]^-{} & & \wt{S}/i
  \enddiagram
  $$
Since $C_t$ is ample on $\wt{S}/i$, we have by Hodge index theorem, that $\wt{C_t}$ is smooth and connected on $\wt{S}$. Consider the following commutative diagram at the level of algebraically trivial zero cycles modulo rational equivalence.

$$
  \diagram
 A_0(C_t)\ar[dd]_-{} \ar[rr]^-{} & & A_0(\wt{C_t}) \ar[dd]^-{j_{t*}} \\ \\
  A_0(\wt{S}/i) \ar[rr]^-{} & & A_0(\wt{S})
  \enddiagram
  $$
Suppose that the pull-back of a generic pencil to $\wt{S}$ is a Lefschetz pencil. Then the fundamental group $\pi_1(\mathcal P\setminus 0_1,\cdots,0_m,t)$ acts irreducibly on $H^1(\wt{C_t},\QQ)$. We have on the other hand $J(C_t)$ mapping to $J(\wt{C_t})$. Call its image as $A_t$. By the equivalence of weight one Hodge structures and abelian varieties there exists a Hodge structure $H_t$ of weight one in $H^1(\wt{C_t},\QQ)$. We prove that this Hodge structure is equivariant under the action of the fundamental group of $\mathcal P\setminus 0_1,\cdots,0_m$. Consider the base extension of $\CC$ to $\CC(x)$, that is to the field with transcendence degree one. The abelian varieties $A_t,J(\wt{C_t})$ are defined over $\Spec \CC$. Consider their pull-back under the morphism
$$\Spec (\CC(x))\to \Spec \CC\;.$$
Call them as $A_{t\CC(x)},J(\wt{C_t})_{\CC(x)}$. Let $L$ be the minimal finite extension of $\CC(x)$, such that these abelian varieties are defined over $L$. Let $\mathcal P'$ be a smooth projective curve with function field $L$. Then we have a finite map $\mathcal P'\to \mathcal P$. Consider the spread of $A_{t\CC(x)}, J(\wt{C_t})_{\CC(x)}$ over some Zariski open $U'$ in $\mathcal P'$. Call them $\bcA,\bcJ$. Then throwing out more points from $U'$ we can assume that $\bcA\to U'$ is smooth and proper and so is for $\bcJ\to U'$. So by Ehresmann's theorem we have a fibration in the sense of complex analytic topology. Therefore $\pi_1(U',t')$, should act on the cohomology of the fibers of $\bcA,\bcJ$. So we have an action of $\pi_1(U',t')$ on $H^{2d-1}(A_t,\QQ), H^1(\wt{C_t},\QQ)$, here $d$ is the dimension of $A_t$ and $t'$ maps to $t$. Now the group $\pi_1(U',t')$ is of finite index in $\pi_1(\mathcal P\setminus 0_1,\cdots,0_m,t)$. By the Picard Lefschetz formula and the finiteness of the index of the subgroup $\pi_1(U',t')$, we have that $H_t=H^{2d-1}(A_t,\QQ)$ is a $\pi_1(\mathcal P\setminus 0_1,\cdots,0_m)$ equivariant subspace in $H^1(\wt{C_t},\QQ)$. Hence $H_t$ is either zero or all of $H^1(\wt{C_t},\QQ)$ by the irreducibility of the monodromy action. Consequently we have that $A_t=0$ or $J(\wt{C_t})$.

Suppose that $A_t=J(\wt{C_t})$. This implies that genus of $C_t$ is greater than or equal to genus of $\wt{C_t}$. On the other hand we have $\wt{C_t}$ mapping dominantly onto $C_t$. So the genus of $C_t$ is less than or equal to genus of $\wt{C_t}$. Therefore genus of $C_t$ is equal to genus of $\wt{C_t}$. But by the Riemann-Hurwitz formula we have that the ramification locus has degree zero. But since $C_t$ is ample, it intersects the ramification locus of $\wt{S}\to \wt{S}/i$ in a set of points. Hence the ramification locus of the finite map $\wt{C_t}\to C_t$ cannot be empty. Therefore $A_t$ cannot be $J(\wt{C_t})$. So the only possibility is that $A_t=0$, in which case we have $C_t$ a rational curve. Therefore $\wt{C_t}$ is hyperelliptic.
\end{proof}

\section{Lefschetz pencils on  threefolds}
If instead of considering a surface and a finite map from the surface onto $\CC\PR^2$, we consider a smooth projective threefold $T$ and a finite map onto $\CC\PR^3$ and define a Lefschetz pencil on $T$, to be a rational map $T\to \CC\PR^3\dashrightarrow\CC\PR^1$, such that the singular fibers can have at most ordinary double points, then we can prove analogs of
the Theorem \ref{theorem1}, Lemma \ref{theorem3}, theorems \ref{theorem4},\ref{theorem8}. The most important in our context is theorem \ref{theorem8}.
Let for $H$ a hyperplane section of $\CC \PR^3$, $S_H$ is the pull-back of $H$ under the finite map from $T$ to $\PR^3$. Let $H^2(S_H,\CC)_{van}=\ker(H^2(S_H,\CC)\to H^4(T,\CC))$.

\begin{theorem}Suppose that $L$ is a very ample line bundle on $\CC\PR^3$ and $f:T\to \CC\PR^3$ is a finite map. Consider the ample line bundle $f^*L$  and the linear system associated to it. Suppose that a generic pencil on $T$ is a Lefschetz pencil. Suppose also that for $S_H\in |f^*L|$, we have $H^{2,0}(S_H)\cap H^2(S_H,\CC)_{van}\neq 0$. Then for a general $S_H$ in the linear system, the restriction map from $\Pic(T)$ to $\Pic(S_H)$ is onto.
\end{theorem}
\begin{proof}
The argument of this theorem follows the idea of the proof of Theorem 3.33 in \cite{Vo}.
Consider the Hilbert schemes $H_{i,U}$, given by
$$\{(Z,H):Z\subset S_H\}\subset H_i\times U$$
where $H_i$'s are Hilbert schemes parametrising one dimensional subschemes on $T$. So what we want to prove that $H_{i,U}$ to $U$ is not dominant unless we have the class of $Z$ from $H_{i,U}$, in $NS(T)|_{S_H}\otimes \QQ$.
Assume that $H_{i,U}$ to $U$ is dominant. We can consider an embedding of $H_{i,U}$ into some projective space and consider a smooth multi-section of it which maps finitely onto $U$ (may be we have to replace $U$ by a smaller Zariski open subset). Consider $H\in U$. Then the fiber $\pr_2^{-1}(H)$ parametrises curves $Z_{1,H},\cdots,Z_{n,H}$ incident on $S_H$. The classes of these curves together with $NS(T)_{S_H}\otimes \QQ$ generate a sub-local system of the local system ${R^2\phi_*\QQ}$. Here $\phi:\bcT_U\to U$ is the universal hypersurface given by
$$\{(x,H):f(x)\in H\}\;.$$
 Call the above local subsystem as $F$. Consider the intersection of $F$ with $R^2\phi_*\QQ_{van}$ (this is the local system of vanishing cohomology groups). Call it $F'$. Since $R^2\phi_*\QQ|_{van}$ is indecomposable by theorem \ref{theorem8}, we have $F'=0$ or all of $R^2\phi_*\QQ|_{van}$. In the second case we have the group $H^2(S_H,\QQ)_{van}$ is generated by divisor classes , which are $(1,1)$ type in the Hodge decomposition. Hence $H^2(S_H,\QQ)_{van}\cap H^{2,0}(S_H)=0$, contradicting our assumption.

 In the first case the classes $Z_{i,H}$ belongs to $NS(T)|_{S_H}\otimes \QQ$.  By the Lefschetz theorem for $(1,1)$ classes the orthogonal complement of $NS(T)|_{S_H}\otimes \QQ$ in the vector space $\langle Z_{i,H}, NS(T)|_{S_H}\otimes \QQ\rangle$  with respect to the intersection pairing on $H^2(S_H,\QQ)$
 is contained in $H^2(S_H,\QQ)_{van}$, hence in $F'$. Since $F'=0$ we conclude that $Z_{i,H}$ is in $NS(T)|_{S_H}\otimes \QQ$.
\end{proof}

\end{document}